\documentclass[11pt]{amsart}
\usepackage[a4paper,margin=2.3cm]{geometry}
\usepackage{mathtools,amsmath,amsthm,amssymb,amscd,amsfonts,color}
\usepackage{booktabs}
\usepackage{enumitem}
\usepackage{tikz}
\usetikzlibrary{arrows,chains,matrix,positioning,scopes,decorations.pathreplacing}
\usepackage{thmtools,thm-restate}
\usepackage{hyperref}
\usepackage{kbordermatrix}
\usepackage{dsfont}
\usepackage{multirow}
\usepackage[noEnd=false]{algpseudocodex}

\theoremstyle{plain}
\newtheorem{theorem}{Theorem}[section]
\newtheorem{proposition}[theorem]{Proposition}
\newtheorem{lemma}[theorem]{Lemma}
\newtheorem{corollary}[theorem]{Corollary}
\numberwithin{equation}{section}

\theoremstyle{definition}

\newtheorem{remark}[theorem]{Remark}
\newtheorem{problem}[theorem]{Problem}
\newtheorem{algorithm}[theorem]{Algorithm}

\newcommand{\C}{\mathbb{C}}
\newcommand{\cP}{\mathcal{P}}
\newcommand{\cS}{\mathcal{S}}

\newcommand{\cF}{\mathcal{F}}
\newcommand{\cR}{\mathcal{R}}

\newcommand{\cK}{\mathcal{K}}
\newcommand{\cX}{\mathcal{X}}
\newcommand{\cG}{\mathcal{G}}
\newcommand{\cO}{\mathcal{O}}
\newcommand{\cZ}{\mathcal{Z}}

\newcommand{\cU}{\mathcal{U}}

\newcommand{\R}{\mathbb{R}}
\newcommand{\Z}{\mathbb{Z}}

\DeclareMathOperator{\Pic}{Pic}

\DeclareMathOperator{\GL}{GL}

\DeclareMathOperator{\MNF}{MNF}

\DeclareMathOperator{\Lk}{Lk}
\DeclareMathOperator{\wed}{wed}
\DeclareMathOperator{\cone}{cone}

\title[Toric colorable PL~spheres of Picard number~$4$]
{The characterization of $(n-1)$-spheres with $n+4$ vertices having maximal Buchstaber number}

\author{Suyoung Choi}
\address{Department of Mathematics, Ajou University, 206, World cup-ro, Yeongtong-gu, Suwon 16499,  Republic of Korea}
\email{schoi@ajou.ac.kr}

\author{Hyeontae Jang}
\address{Department of Mathematics, Ajou University, 206, World cup-ro, Yeongtong-gu, Suwon 16499, Republic of Korea}
\email{a24325@ajou.ac.kr}

\author{Mathieu Vall\'ee}
\address{Université Sorbonne Paris Nord, LIPN, CNRS UMR 7030, F-93430, Villetaneuse, France}
\email{vallee@lipn.fr}

\date{\today}


\thanks{This work was supported by the National Research Foundation of Korea Grant funded by
the Korean Government (NRF-2019R1A2C2010989).}

\begin{document}
\begin{abstract}
    We present a computationally efficient algorithm that is suitable for graphic processing unit implementation. This algorithm enables the identification of all weak pseudo-manifolds that meet specific facet conditions, drawn from a given input set. We employ this approach to enumerate toric colorable seeds. Consequently, we achieve a comprehensive characterization of $(n-1)$-dimensional PL~spheres with $n+4$ vertices that possess a maximal Buchstaber number.

    A primary focus of this research is the fundamental categorization of non-singular complete toric varieties of Picard number~$4$.
    This classification serves as a valuable tool for addressing questions related to toric manifolds of Picard number~$4$.
    Notably, we have determined which of these manifolds satisfy equality within an inequality regarding the number of minimal components in their rational curve space.
    This addresses a question posed by Chen, Fu, and Hwang in 2014 for this specific case.
\end{abstract}
{\let\newpage\relax\maketitle}

\section{Introduction}
Our interest lies at the intersection of geometry, with the classification of non-singular complete toric varieties, and discrete mathematics, with the enumeration of piecewise linear (PL) spheres.
\subsection*{Toric geometry}
A \emph{toric variety} of complex dimension $n$ is a normal algebraic variety over the field of complex numbers $\C$ that admits an effective algebraic action of $(\C^\ast)^n$ having a dense orbit.
The \emph{fundamental theorem for toric geometry} states that the classification of toric varieties of complex dimension $n$ is equivalent to that of \emph{fans} in $\R^n$.
The \emph{cone} generated by a finite set of rational vectors $R \subset \R^n$ is $\cone(R) = \{\sum_{r\in R} \alpha_r r\colon \alpha_r\geq 0\}$.
A fan in $\R^n$ is a collection of cones that is closed under taking faces and such that the intersection of any pair of them is a face of both.
The one-dimensional cones of a fan are called the \emph{rays}.
The combinatorial structure of a fan $\Sigma$ in $\R^n$ having $m$ rays is represented by a pair $(K,\lambda)$ where:
\begin{itemize}
	\item $K$ is the underlying complex of $\Sigma$, with vertex set $[m]=\{1,\dots,m\}$, whose face lattice is isomorphic to that of $\Sigma$, and
	\item $\lambda \colon [m]\to\Z^n$ is a map that is one-to-one assigning vertices of $K$ to the primitive generator of the rays of $\Sigma$.
\end{itemize}
The fan is therefore given by $\Sigma = \{\cone \left( \{\lambda(i)\}_{i\in \sigma}\right)\colon \sigma\in K\}$.
In this article, we are interested in compact smooth toric varieties, simply \emph{toric manifolds}, that are characterized by \emph{complete non-singular} fans, whose pairs $(K,\lambda)$ satisfy:
\begin{itemize}
	\item $K$ is a simplicial complex, which is a PL~sphere, and
	\item $\lambda$ satisfies the \emph{non-singularity condition} for $K$: for any face $\{i_1, \ldots, i_k\}$ in $K$, the set $\{\lambda(i_1), \ldots, \lambda(i_k)\}$ is \emph{unimodular}, namely it is part of a basis of $\Z^n$.
\end{itemize}
For a simplicial complex $K$ on $[m]$ with $\dim(K)=n-1$, its \emph{Picard number} is $\Pic(K)=m-n$.
When $K$ is obtained from a complete non-singular fan, this number is the \emph{Picard number} of the associated toric manifold, see~\cite[Section~3.4]{Fulton1997book}.
Kleinschmidt~\cite{Kleinschmidt1988} and Batyrev~\cite{Batyrev1991} classified toric manifolds of ``small" Picard numbers~$2$ and $3$, respectively.

The classification of complete non-singular fans necessitates the classification of PL~spheres that can serve as their underlying simplicial complexes.
For Picard number~$\leq 3$, every PL~sphere is polytopal by Mani~\cite{Mani1972}, and can therefore be described using Gale diagrams~\cite{Grunbaum2003}.
Using this, one can characterize which PL~spheres support a complete non-singular fan~\cite{Gretenkort-Kleinschmidt-Sturmfels1990,Erokhovets2011}.

In this article, we take one more step and focus on the case of Picard number~$4$.
However, the same method is hardly applicable for Picard number~$4$ since $3$-dimensional Gale diagrams are difficult to use.
Moreover, a non-polytopal PL~sphere of Picard number~$4$ exists as shown in~\cite{Grunbaum-Sreedharan1967}, and there also exists a complete non-singular fan whose underling simplicial complex can be non-polytopal~\cite{Suyama2014}.
Therefore, we approach the problem using other combinatorial properties of PL~spheres, such as their property of being \emph{weak pseudo-manifolds}.

\subsection*{Enumeration of weak pseudo-manifolds and PL~spheres}
The enumeration of triangulations of manifolds has been a longstanding challenge since the end of the 19th century. The advances made in this area have provided valuable tools for researchers studying discrete and PL~geometry. Computer-assisted enumeration has been a major approach for tackling these problems and we follow this direction in the present article.

In particular, the aforementioned works focus on the enumeration of weak pseudo-manifolds, pseudo-manifolds, PL~manifolds, PL~spheres, and polytopal spheres.
Let $K$ be a simplicial complex on $[m]$.
It is \emph{pure} if its maximal faces are all of the same size.
These top dimensional faces are called the \emph{facets} of $K$, and are of size $\dim(K)+1$.
The faces of size $\dim(K)$ are called the \emph{ridges}.
A simplicial complex is a \emph{weak pseudo-manifold} if it is pure and every ridge is contained in exactly two facets.
Additionally, it is a \emph{pseudo-manifold} (without boundaries) if its ridge-facet graph is connected.
One example of a pseudo-manifold is the boundary of the $(n-1)$-simplex $\partial \Delta^{n-1}$ whose facets are the subsets of size $n$ of $[n+1]$, and has Picard number~$1$.
Any set of affinely independent points $v_1,\ldots,v_{n+1}\in\R^n$ yields a geometric realization $|\partial \Delta^{n-1}|$ of $\partial \Delta^{n-1}$ that is homeomorphic to the sphere $S^{n-1}$.
A simplicial complex $K$ of dimension $n-1$ is a PL~sphere if there exists a subdivision of $K$ and a subdivision of $\partial \Delta^{n-1}$ such that these subdivisions are isomorphic.
It is a \emph{PL~manifold} if the link of each of its faces is a PL~sphere.
A \emph{polytopal sphere} is the boundary complex of a simplicial polytope.
We have the following hierarchy on simplicial complexes:
\begin{center}
	\begin{tikzpicture}[node distance=0.165\textwidth]
	\node[,text width=1.8cm] (A) {polytopal spheres};
	\node[right of=A,text width=2cm] (B) {PL~spheres};
	\draw[draw=white] (A) -- node {$\subseteq$} (B);
	\node[right of=B] (C) {PL~manifolds};
	\draw[draw=white] (B) -- node {$\subseteq$} (C);
	\node[right of=C,text width=2cm] (D) {pseudo-manifolds};
	\draw[draw=white] (C) -- node {$\subseteq$} (D);
	\node[right of=D,text width=2cm] (E) {weak pseudo-manifolds};
	\draw[draw=white] (D) -- node {$\subseteq$} (E);
	\node[right of=E,text width=2cm] (F) {pure\\ simplicial complexes.};
	\draw[draw=white] (E) -- node {$\subseteq$} (F);
\end{tikzpicture}
\end{center}
The notable advance in enumerating these simplicial complexes progresses along two directions: small dimensions and small Picard numbers.
It is well-known that PL~spheres of dimension~$2$ are equivalent to $3$-connected planar graphs, and they can be generated using the \texttt{plantri} algorithm, as demonstrated in~\cite{Brinkmann2007}, for up to $23$ vertices.
Enumerations of polytopal or non-polytopal PL~spheres of dimension~$3$ with $8$ and $9$ vertices are provided in~\cite{Altshuler_Steinberg_1976, Altshuler1985}, following the work in~\cite{Grunbaum-Sreedharan1967}, and~\cite{Altshuler_Bokowski_Steinberg_1980}, respectively.
While the enumeration of PL~spheres of Picard number~$\leq 3$ has been accomplished in \cite{Grunbaum2003}, the enumeration for Picard number~$4$ remains an open problem.

The enumerations of all weak pseudo-manifolds of dimension~$2$ with $7$ and $8$ vertices are documented in~\cite{Datta_1999} and~\cite{Datta_Nilakantan_2002}, respectively.
Lutz and Sulanke extensively used a new method based on lexicographic enumeration to obtain (weak) pseudo-manifolds of dimensions~$2$ and~$3$, with up to $12$ and $11$ vertices, as detailed in~\cite{Lutz2008,Lutz2009}.
Additionally, in~\cite{Bagchi_Datta_1998}, a characterization of pseudo-manifolds of Picard number~$\leq 3$ is provided.

We challenge the enumeration of PL~spheres and weak pseudo-manifolds of Picard number~$4$.
In this paper, we introduce a new method that consists in representing a pure simplicial complex as a $\{0,1\}$-vector, allowing for the use of linear algebra for fast computations, and adaptability to graphic processing unit (GPU) programming.
In Section~\ref{section:graphic processing unit algorithm}, readers can find an explicit algorithm described in the Compute Unified Device Architecture (CUDA) language for enumerating weak pseudo-manifolds whose facets are in a given input set and satisfy affine conditions on the associated vector.


\subsection*{(Real) Buchstaber number and classification problems in toric topology}
Without any assumption on the simplicial complex $K$ on $[m]$, we construct a topological space called the \emph{polyhedral product} $ (\underline{X}, \underline{Y})^K $ of $K$ with respect to a pair $(X, Y)$ of topological spaces which is
$$
(\underline{X}, \underline{Y})^K \coloneqq \bigcup_{\sigma \in K} \left\{ (x_1, \ldots, x_m) \in X^m \mid x_i \in Y \text{ when } i \notin \sigma \right\}.
$$
The \emph{moment-angle complex} $\cZ_K$ of $K$ is $(\underline{D^2}, \underline{S^1})^K$, and the \emph{real moment-angle complex} $\R\cZ_K$ of $K$ is $(\underline{D^1}, \underline{S^0})^K$, where $D^d$ represents the $d$-dimensional disk, and $S^{d-1}$ denotes its boundary sphere of dimension $d-1$.
We observe that the $T^1$-action on the pair $(D^2, S^1)$ yields the canonical action of the $m$-dimensional torus $T^m = (S^1)^m$ on $\cZ_K$.
The \emph{Buchstaber number} $s(K)$ is the maximal integer $r$ for which there exists a subtorus of rank $r$ acting freely on $\cZ_K$.

Similarly, there is an $S^0$-action on the pair $(D^1, S^0)$.
For clarity and consistency in our terminology throughout this paper, we will treat $S^0$ as the additive group $\Z_2 = \Z/2\Z$ with two elements $\{ 0, 1\}$.
This yields the canonical $\Z_2^m$-action on $\R\cZ_K$.
The \emph{real Buchstaber number} $s_\R(K)$ is defined by the maximal rank $r$ of a subgroup of $\Z_2^m$ acting freely on $\R\cZ_K$.
Determination of the Buchstaber and real Buchstaber numbers is one of the central questions in toric topology, and it has been actively studied in many literatures such as in \cite{Fukukawa2011, Erokhovets2014, Hasui2015, Ayzenberg2016, Baralic2022}.

It is noteworthy that when an $r$-dimensional subtorus $H$ of $T^m$ acts freely on $\cZ_K$, the resulting quotient space $\cZ_K /H$ supports a well-behaved torus action $T^m/H \cong T^{m-r}$ with an orbit space that exhibits a reverse face structure isomorphic to $K$.
It is known that $s(K)\leq s_\R(K) \leq \Pic(K)$, see \cite{Erokhovets2014}.
In particular, a toric manifold associated with $K$ is topologically obtainable from the quotient of $\cZ_K$ by a free action of subtorus of dimension $m-n$, see \cite{Davis-Januszkiewicz1991,Buchstaber-Panov2015} for details.
Consequently, considering PL~spheres whose Buchstaber number is maximal, that is equal to $m-n$, is of significant importance.

A simplicial complex of dimension $n-1$ is said to be \emph{toric colorable} if it has a maximal Buchstaber number, and $\Z_2^n$-\emph{colorable} if it has a maximal real Buchstaber number.
We have the following hierarchy on PL~spheres:
\begin{center}
	\begin{tikzpicture}[node distance=0.2\textwidth]
	\node(A) {fan-like};
	\node[right of=A] (B) {toric colorable};
	\draw[draw=white] (A) -- node {$\subseteq$} (B);
	\node[right of=B] (C) {$\Z_2^n$-colorable,};
	\draw[draw=white] (B) -- node {$\subseteq$} (C);
\end{tikzpicture}
\end{center}
where a \emph{fan-like} PL~sphere is a simplicial complex that is the underlying complex of some complete non-singular fan.
Therefore, the first step for classifying toric manifolds of Picard number~$4$ is to characterize toric colorable PL~spheres of Picard number~$4$.
This derives the following classification problem.

\begin{problem} \label{problem:main}
	Find which PL~spheres of Picard number~$4$ have (real) Buchstaber number~$4$.
\end{problem}

The key tool for obtaining the answer lies in the finiteness of the problem.
This result stems from Choi and Park~\cite{CP_wedge_2}, who established that there exists a finite set of PL~spheres, called toric colorable \emph{seeds}, from which all toric colorable PL~spheres can be derived through iterated \emph{wedge operations} (this operation is also referred to as the \emph{$J$-construction} in~\cite{BBCG2010}).
The maximal number of vertices of a toric colorable seed of Picard number $p\geq 3$ is $2^p-1$.
Consequently, we only need to enumerate the toric colorable seeds up to $n=11$.

In this paper, we use that toric colorable seeds must belong to certain binary matroids, as detailed in Section~\ref{section:preparation}.
This restricts the number of facets inputted into our GPU algorithm, allowing us to obtain results for $n$ up to $10$.
In addition, we mathematically address the extreme case $n=11$ to further reduce the algorithmic complexity and derive the following main theorem.
\begin{restatable*}{theorem}{main} \label{thm:main}
	Up to isomorphism, the number of toric (or $\Z_2^n$-)colorable seeds of dimension~$n-1$ and Picard number~$p\leq 4$ is as follows:
	\begin{center}
		\begin{tabular}{l*{12}{c}r}
			\toprule
			$p\backslash\:\: n$ &$1$ & $2$ & $3$ & $4$ & $5$ & $6$ & $7$ & $8$ & $9$ & $10$ & $11$ & $>11$& total \tabularnewline \midrule
			$1$&$1$&  &&&&&&&&&&&$1$\tabularnewline
			$2$ && $1$ & & &&&&&&&&& $1$\tabularnewline
			$3$ && $1$ & $1$& $1$& & & & &&&& &$3$\tabularnewline
			$4$ && $1$ & $4$ & $21$ & $142$ & $733$ & $1190$ & $776$ & $243$& $39$ & $4$ & &$3153$\tabularnewline
			\bottomrule
		\end{tabular}
	\end{center}
	with the empty slots displaying zero.
\end{restatable*}

The database containing the toric colorable seeds of Picard number~$4$, the CUDA script, and a C++ version of the script (which we used for both performance comparison and verification purposes) are available on the third author's Github repository:
\begin{center}
\url{https://github.com/MVallee1998/GPU_handle}
\end{center}

From~\cite{CP_wedge_2}, we obtain this corollary of Theorem~\ref{thm:main} which completely solves Problem~\ref{problem:main}.
\begin{corollary}\label{corollary:main} The toric (or $\Z_2^n$-)colorable PL~spheres of dimension $n-1$ and Picard number~$p\leq 4$ are precisely those obtained after consecutive wedge operations on the toric (or $\Z_2^n$-)colorable seeds (of Theorem~\ref{thm:main}).
\end{corollary}
In summary, the set of toric (or $\Z_2^n$-)colorable PL~spheres of dimension $n-1$ and Picard number~$p\leq 4$ is ``finitely generated" through multiple wedge operations on the explicit  $1+1+3+3153$ seeds outlined in Theorem~\ref{thm:main}.


\subsection*{Application to the space of rational curves on toric manifolds}
We anticipate that our theorem contributes to the understanding of toric manifolds of Picard number~$4$.
For instance, in this paper, we employ it to address a question posed by Chen, Fu, and Hwang\cite{fu2014minimal} in 2014 for this specific case.

Let $X$ be a toric manifold whose corresponding fan has $m$ rays and $\text{RatCurves}(X)$ the normalized space of rational curves on $X$.
Fix an irreducible component $\cK$ of $\text{RatCurves}(X)$.
Then we have a universal family $\rho \colon \cU \longrightarrow \text{RatCurves}(X)$, which is a complex projective line bundle, and $\mu \colon \cU \longrightarrow X$, which is an evaluation map.
The component $\cK$ is called \emph{minimal} if $\mu$ is dominant and $\mu^{-1}(x)$ is complete for a general point $x \in X$.
The \emph{degree} of $\cK$ is defined as the degree of the intersection of the anti-canonical divisor of $X$ with any member in $\cK$.
In \cite{fu2014minimal}, it is shown that the sum of the degrees of all minimal components is less than or equal to $m$, and it is asked when the equality holds.
In Section~\ref{sec:min_comp} we answer this question for every toric manifold of Picard number~$4$.

\section{Classification of weak pseudo-manifolds by graphic processing unit computing} \label{section:graphic processing unit algorithm}
In this section, we provide a general approach on how to use graphic processing unit (GPU) parallel computing capability for classifying weak pseudo-manifolds with given properties.

\subsection{Enumerating weak pseudo-manifolds}\label{subsection:enum_weak_psdmfd}
Let $K$ be a pure simplicial complex of dimension $n-1$ on the vertex set $[m] = \{1, 2, \ldots, m\}$.
A \emph{facet} of $K$ is a face of size $n$, and a \emph{ridge} is a face of size $n-1$.
We denote by $\cF(K)$ and $\cR(K)$ the sets of facets and ridges of $K$, respectively.
We will often use the words facets and ridges without specifying a simplicial complex when referring to a subset of size $n$ and a subset of size $n-1$ of $[m]$.
Technically, they refer to the facets and ridges of the simplicial complex whose facets are the subsets of size $n$ of $[m]$.
We shall provide an algorithm as follows:

\begin{algorithmic}
    \algrenewcommand\algorithmicrequire{\textbf{Input:}}
	\algrenewcommand\algorithmicensure{\textbf{Output:}}
	\Require{ A set $\cF$ of facets, and a collection $\cG$ of affine functions on the subsets of $\cF$, called \emph{properties}.}
	\Ensure{ The set of weak pseudo-manifolds $K$ such that $\cF(K)\subseteq \cF$ and $g(\cF(K))>0$ for all $g\in\cG$, namely, $K$ \emph{satisfies} all the properties.}
\end{algorithmic}

Provided any set of facets $\cF = \{F_1,\ldots,F_M\}$, we can compute the set $\cR=\{r_1,\ldots,r_N\}$ of all ridges that come from these facets.
We then construct the \emph{ridge-facet incidence matrix} $A(\cF) = (a_{i,j})$ of size $N \times M$ as follows:

$$a_{i,j} = \begin{cases} 1 & r_i\subset F_j\\ 0 & \text{otherwise}\end{cases},$$
for $i=1,\ldots,N$ and $j=1,\ldots,M$.
A simplicial complex $K$ whose facets are all in some set of facets $\cF=\{F_1,\ldots,F_M\}$ can be regarded as a \emph{characteristic vector} $K =(k_1, \ldots, k_M)^t \in\Z^M$ with

$$k_j = \begin{cases} 1 & F_j\in K\\ 0 &  F_j\notin K\end{cases},$$
for $j=1,\ldots,M$.
The pure simplicial complex $K$ is a \emph{weak pseudo-manifold} if any ridge of $K$ is in exactly two facets of $K$. That reflects in the following property:
\begin{proposition}\label{proposition:weak_pseudo-mfd}
	Let $\cF$ be a set of facets, $A=A(\cF)$ the ridge-facet incidence matrix of $\cF$, and $K$ a pure simplicial complex whose facets are all in $\cF$. Then $K$ is a weak pseudo-manifold if and only if the coordinates of the product $AK$ are all in $\{0,2\}$.
\end{proposition}

From that, in $\Z_2^M$, the characteristic vectors of weak pseudo-manifolds are all included in the $\Z_2$-kernel of the matrix $A$, seen as a linear map $A\colon \Z_2^M \to \Z_2^N$.

Let $B=\begin{bmatrix} K_1 &\cdots & K_s\end{bmatrix}$ be a matrix whose columns form a $\Z_2$-basis of $\ker_{\Z_2} A$.
Every weak pseudo-manifold $K$ is uniquely expressed as one of the $2^s$ possible $\Z_2$-linear combinations of $K_1,\ldots,K_s$, namely $K=\sum_{i=1}^s x_i K_i \pmod{2}= BX$, for $X=(x_1,\ldots,x_s)^t\in\Z_2^s$.
We find a suitable basis $\widetilde{K}_1,\ldots,\widetilde{K}_s$ to reduce the number of cases to compute.

We first explain how to construct this basis when the set $\cF=\binom{[m]}{n}$ contains all possible facets of $[m]$ and $\cR=\binom{[m]}{n-1}$ all the ridges.
There are $\binom{m}{n}$ facets and $\binom{m}{n-1}$ ridges.
For a ridge $r$, we will write as $(AK)_r$ the coordinate of $AK$ corresponding to $r$.
Let us denote by $\cP(r):=\{j\in[M] \colon r\subset F_j\}$ the set of the indexes in $\cF$ of the facets containing $r$, called the \emph{parents} of $r$, that are the only facets contributing to $(AK)_r$.
In this first case, any ridge has $m-n+1$ parents.
For a kernel matrix $B$ whose row are indexed by $\cF$, let us denote by $B_{\cP(r)}$ the matrix whose rows are the ones of $B$ taken at indexes $\cP(r)$.
For every $r\in\cR$, for every $t=1,\ldots,s$, the $t$th column of $B_{\cP(r)}$ has an even number of ones since the basis element $K_t$ has an even number of facets containing $r$.
Performing a mod~$2$ Gaussian elimination on the columns of $B_{\cP(r)}$ yields a matrix of the following form

$$B_{\cP(r)} E=\begin{bmatrix}Z_{m-n}&\mathbf{0}\end{bmatrix},$$
with the $(k+1)\times k$-matrix

$$ Z_k = \begin{bmatrix}
	1&0&\cdots&0\\
	0& 1 & \ddots &\vdots\\
	\vdots&\ddots&\ddots&0\\
	0&\cdots&0&1\\
	1&\cdots&1&1
\end{bmatrix},$$
for some integer $k$, and $E\in \GL(s,\Z_2)$ corresponding to the operations performed in the Gaussian elimination.
The columns of the new matrix $B E$ correspond to a convenient basis of the $\Z_2$-kernel of $A$.
Indeed, only its first $m-n$ columns have facets contributing to $(AK)_r$.
Moreover, taking the mod~$2$ linear combination of strictly more than two of them makes $(AK)_r$ be strictly greater than $2$, which is a case we want to avoid computing since we focus on weak pseudo-manifolds, see Proposition~\ref{proposition:weak_pseudo-mfd}.
Thus, this decreases the number of $\Z_2$-linear combinations containing the first $m-n$ new generators that we need to compute from $2^{m-n}$ to $\binom{m-n}{0}+\binom{m-n}{1}+\binom{m-n}{2}$.

By writing $r^1:=r$ and $E_1:=E$, one can inductively repeat the latter process by taking care at step $k+1$ of:
\begin{itemize}
	\item choosing each time a new ridge $r^{k+1}$ such that for all $i=1,\ldots,k,\cP(r^i)\cap\cP(r^{k+1})=\emptyset$, and
	\item starting the Gaussian pivot at columns index $k(m-n)+1$ so that the structure of the generators of previous columns is not lost.
\end{itemize}
This process terminates at some step $k_{\max}$ whenever one of the former conditions cannot be satisfied.
We obtain a final matrix, whose columns are the new basis elements $\widetilde{K}_1,\ldots,\widetilde{K}_s$, and, up to reordering, whose rows are according to the sets $\cP(r^1),\ldots,\cP(r^{k_{\max}})$ looks as follows:

$$BE_1\cdots E_{k_{\max}} = \begin{bmatrix}
	Z_{m-n}& 0 & \hdots & \hdots & 0 \\
	0 & Z_{m-n} & \ddots &  & \vdots \\
	\vdots & \ddots & \ddots & \ddots & \vdots \\
	0 & \hdots & 0 &Z_{m-n} & 0 \\ \hline
	\star & \star & \star & \star & \star \\
\end{bmatrix} = \begin{bmatrix}
	\widetilde{K}_1&\cdots&\widetilde{K}_s
\end{bmatrix}.$$
In this case, we decrease the total number of $\Z_2$-linear combinations from $2^s$ to $(1+(m-n) + \binom{m-n}{2})^{k_{\max}} 2^{s-k_{\max}(m-n)}$ since we should take at most $2$ basis elements for each block $Z_{m-n}$ and there remains $s-k_{\max}(m-n)$ columns $\widetilde{K}_i$ after these blocks.

As for the general case, there may be ridges having less than $m-n+1$ parents.
In this case, we try to wisely choose some ridges $r^1,\ldots,r^{k_{\max}}$ such that the blocks $Z_k$ are of the maximum possible size so we minimize the number of $\Z_2$-linear combinations $BX$ of the generators we need to compute.
That provides a partition $I_1,\ldots,I_l$ of $\{1,\ldots,s\}$ such that if we are to sum more than two basis elements with indexes in $I_k$, for $k=1,\ldots,l$, we are sure not to obtain a weak pseudo-manifold.
We can split the vector $X$ in the $\Z_2$-linear combinations $BX$ as blocks according to this partition: $X =\sum_{k=1}^l X_k$, with $X_k$ representing the part of $X$ whose only nonzero coordinates are in $I_k$.
Let us denote by $\cX_k$ the set of all such possible $X_k$ for $k=1,\ldots,l$.

If we recap our process, given a set of facets $\cF$, we constructed
\begin{itemize}
	\item the ridge-facet incidence matrix $A$ whose $\Z_2$-kernel contains all weak pseudo-manifolds,
	\item a matrix $B$ whose columns form a convenient basis $\widetilde{K}_1,\ldots,\widetilde{K}_s$ of $\ker_{\Z_2}(A)$,
	\item a partition $I_1,\ldots,I_l$ of $\{1,\ldots,s\}$,
	\item the sets $\cX_1,\ldots,\cX_l$ of partitions of the vectors of $\Z_2^s$ such that for all $k=1,\ldots,l$, $X_k\in\cX_k$ has a maximum of two nonzero coordinates which are all in $I_k$,
\end{itemize}
such that any weak pseudo-manifold whose facets are in $\cF$ is of the form $K = BX$, with $X=\sum_{k=1}^l X_k$ for some $(X_1,\ldots,X_l)\in \cX(\cF) =\cX_1\times\cdots\times \cX_l$, satisfying that the entries of $AK$ corresponding to the chose ridges are in $\{0,2\}$.
Moreover, given any affine function $K\mapsto g(K)$, it is easy to check using computer programming that $g(K)>0$ is verified.


\subsection{Generalities about GPU programming}\label{subsection:cuda}
In this article, we used \emph{Nvidia Compute Unified Device Architecture} (CUDA) \cite{cuda}.
One will notice that the syntax and vocabulary may differ from other GPU languages.

The general idea behind GPU computing is that it allows parallelizing tasks with two layers of parallel programming without requiring a supercomputer.
Parallel programming takes several forms, and the two we will use are the following.
\begin{itemize}
	\item Data parallelism: one has a list of elements $\cX$ and wants to apply the same function $g$ to every element $X\in\cX$. In this case, each call of the function $g$ is independent.
	\item Task parallelism: one has an element $X$ and wants to apply a set of similar functions $g_1,\ldots,g_k$ on $X$ in order to obtain the result as a list $(g_1(X),\ldots,g_k(X))$.
	The simplest example is a matrix product $AX$, and if each row of $A$ is denoted by $a_i$, then the functions $g_i$ are the inner products with the $a_i$s.
\end{itemize}

In all that follows, a \emph{thread (of execution)} will be a processing unit that computes machine operations linearly, and a GPU will be a two-layered structure of threads.
Namely, a GPU will be a set of $p$ \emph{grids}, and each grid will be a set of $q$ threads.
Therefore a GPU can be seen as $p\times q$ threads organized for parallel programming, as in Figure~\ref{figure:GPU}.
The number $p\times q$ of GPU threads that can run simultaneously is roughly the number of CUDA cores (if we consider Nvidia GPUs) and is around seventeen thousand for the current architectures (as of 2024).
Thus a single GPU would be approximately equivalent to at least a thousand central processing unit (CPU) threads.

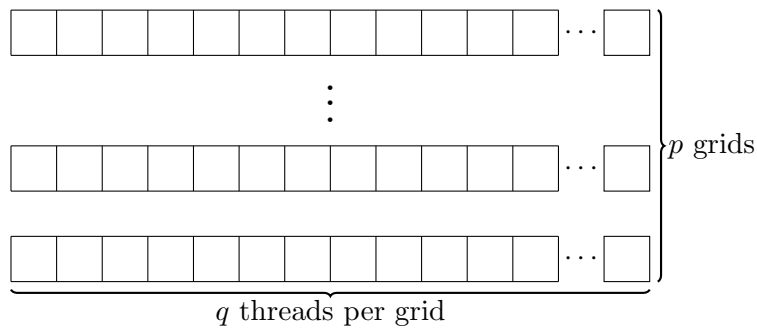
\begin{figure}[ht]
	\begin{tikzpicture}[scale =0.6]
		
		\draw (0,0) -- (12,0);
		\draw (0,1) -- (12,1);
		\foreach \x in {0,1,...,12} \draw (\x,0) -- (\x,1);
		\draw (13,0) -- (14,0) -- (14,1) -- (13,1) -- cycle;
		\draw (12.55,0.5) node {$\cdots$};
		\draw [decorate,decoration = {brace,mirror},thick,yshift = -5pt] (0,0) -- node[below] {$q$ threads per grid} (14,0);
		
		\draw (0,2) -- (12,2);
		\draw (0,3) -- (12,3);
		\foreach \x in {0,1,...,12} \draw (\x,2) -- (\x,3);
		\draw (13,2) -- (14,2) -- (14,3) -- (13,3) -- cycle;
		\draw (12.55,2.5) node {$\cdots$};
		
		\draw (7,4.2) node[scale=1.5] {$\vdots$};
		
		\draw (0,5) -- (12,5);
		\draw (0,6) -- (12,6);
		\foreach \x in {0,1,...,12} \draw (\x,5) -- (\x,6);
		\draw (13,5) -- (14,5) -- (14,6) -- (13,6) -- cycle;
		\draw (12.55,5.5) node {$\cdots$};
		
		\draw [decorate,decoration = {brace,mirror},thick,xshift = 5pt] (14,0) -- node[right] {$p$ grids} (14,6);
	\end{tikzpicture}
	\caption{The two layered parallel structure of a GPU.}
	\label{figure:GPU}
\end{figure}

In CUDA programming we use this two-layered structure as follows:
\begin{itemize}
	\item \textbf{First layer (blocks):} Let $\cX=\{X_1,\ldots,X_N\}$ be the set of data on which we want to apply the same function $g$, called the \emph{kernel}.
	We create some list of $N$ \emph{blocks} indexed by an integer $i$.
	Each block embodies the function call $g(X_i)$.
	A block has three possible states: \emph{on hold}, \emph{active}, and \emph{completed}.
	In the beginning, every block is on hold.
	Then the $p$ grids of the GPU are filled with some blocks which will be running, these are active, and the rest are waiting to be launched on the grid and remain on hold.
	Whenever some active block has completed, the GPU replaces it with a block on hold.
	The program terminates when all blocks are completed.
	\item \textbf{Second layer (threads):} Whenever we send a block to a grid, the operations made in the block are split into threads using task parallelism, and we distribute any procedures in $g$ into $q$ functions which will run simultaneously on all $q$ threads of the grid.
	Notice that we need every thread to finish its tasks to obtain the result.
	We can explicitly require this condition by \emph{synchronizing} the threads.
\end{itemize}
\medskip

In all that follows, we will use such notations:
\begin{itemize}
	\item a set $\cX$ is denoted as a list \texttt{list\_X},
	\item a matrix $A = [a_{i,j}]$ is represented as an array \texttt{A} whose entry at index $i,j$ is \texttt{A[i][j]}$=a_{i,j}$, and
	\item a binary vector $X\in\Z_2^k$ is represented as a binary variable \texttt{x} on $k$ bits.
\end{itemize}

We will use the following processor instructions on binary variables \cite{cuda}:
\begin{itemize}
	\item the \emph{bitwise and} operation \texttt{x\&y}, 64 operations per cycle,
	\item the \emph{bitwise exclusive-or} operation \texttt{x\^{}y}, 64 operations per cycle,
	\item the \emph{population count} operation \texttt{popcount(x)} which counts the number of ``1" bits, called \emph{active} bits, in the value of \texttt{x}, 16-32 operations per cycle, and
	\item \emph{atomic operations}, that we use to avoid memory access errors when many threads may want to write at the same memory location concurrently. The processor scheduler creates a queue of all atomic operation calls.
\end{itemize}
A \emph{cycle} is the shortest time interval considered in a processor unit that it performs at its frequency.
If the frequency is 1GHz the processor realizes $10^9$ cycles per second.

The thread synchronization allows us to manage how the threads behave in parallel as follows:
\begin{itemize}
	\item The \texttt{syncthreads()} command asks all the threads to wait for each others and come across the same line in the algorithm code of the kernel.
	\item For a local thread variable \texttt{test}, the \texttt{syncthreads\_and(test)} and \texttt{syncthreads\_or(test)} command allows us to manage the \textit{and} and the \textit{or} operation over all of the variables \texttt{test} existing in each thread of a grid.
	For example, if a thread encounters a condition that should stop the current case in a loop, then all the threads should stop at once since it is useless to compute this case.
\end{itemize}

\subsection{The GPU algorithm for classifying weak pseudo-manifolds}\label{subsection:gpu_algo}

To simplify our explanations, we suppose that there are $s=64$ generators and that we can write the product $\cX_1\times\cdots\times \cX_l$ as $\cX_a\times \cX_b$ such that $\cX_a$ and $\cX_b$ describe the 32 first and last generators, respectively. We thus decompose $K$ as $K_a+K_b$, with $K_a = B X_a$ and $K_b = B X_b$ for every $(X_a,X_b)\in\cX_a\times\cX_b$.
Both vectors $X_a$ and $X_b$ are binary vectors whose nonzero coordinates are in the $32$ first or last coordinates, respectively, which we store as $32$ bits variables \texttt{xa} and \texttt{xb}, namely as unsigned integers.

The dot product in $\Z$ of the binary forms \texttt{x} and \texttt{y} of two integers $x$ and $y$, respectively, is the number \texttt{popcount(x\&y)} of active bits of the \emph{bitwise and} operation.
Its mod~$2$ reduction is \texttt{popcount(x\&y)\&1}, the value of its least significant bit.

The main idea of the algorithm is to use $M$ threads to compute each coordinate of $K\in\Z_2^M$, with $M$ being the number of facets in $\cF$, as provided in Algorithm~\ref{algorithm:main}.
\begin{algorithm}\label{algorithm:main}
	The GPU algorithm for enumerating weak pseudo-manifolds.
	
	\begin{algorithmic}[1]
		\algrenewcommand\algorithmicrequire{\textbf{Input:}}
		\algrenewcommand\algorithmicensure{\textbf{Output:}}
		\algrenewcommand\algorithmicprocedure{\textbf{Procedure}}
		\algrenewcommand\algorithmicfunction{\textbf{Function}}
		\Require{The list \texttt{list\_F}, corresponding to the set of facets $\cF$, and the list \texttt{list\_G}, corresponding to the set of affine functions $\cG$.}
		\Ensure{The list \texttt{list\_K} of weak pseudo-manifolds \texttt{K} with facets in \texttt{list\_F} and that satisfy \texttt{g(K)>0} for every \texttt{g} in \texttt{list\_G}.}
		\Procedure{Initialization}{}
		\State Compute the ridge-facet incidence matrix $A=A(\cF)\in\Z_2^{N\times M}$ and store it in \texttt{A}, a column sparse matrix: \texttt{A[k][i]} represents the index of the \texttt{k}th nonzero coordinate of the \texttt{i}th column of $A$.
		\State 	Compute $B = \begin{bmatrix}\widetilde{K}_1 & \cdots & \widetilde{K}_{64}\end{bmatrix} = \begin{bmatrix}a_1 & b_1\\ \vdots & \vdots \\ a_M & b_M\end{bmatrix}$ and store it as two lists \texttt{list\_a} and \texttt{list\_b} of integers, where \texttt{list\_a[k]} and \texttt{list\_b[k]} represents the binary value of the row vectors $a_k$ and $b_k$, respectively.
		\State 	Enumerate $\cX_a$ and $\cX_b$, and store them as two lists \texttt{list\_Xa} and \texttt{list\_Xb}.
		\State 	Create a list \texttt{list\_Ka} of all the $K_a$s
		\ForAll{\texttt{xa} $\in$ \texttt{list\_Xa}}
		\For{\texttt{k} $=1,\ldots, M$}
		\State \texttt{Ka[k]}$\gets$\texttt{popcount(a[k]\&{}xa)\&1}
		\EndFor
		\EndFor
		\EndProcedure
		\State \textbf{Shared memory:}
		Integer array \texttt{r} of size $N$, such that \texttt{r[k]} stores the \texttt{k}th coefficient of the product $AK$.
		\Function{Kernel}{\texttt{xa,Ka}}
		\State \texttt{i}$\gets$local thread index
		\State \texttt{b}$\gets$\texttt{list\_b[i]}
		\State \texttt{ka}$\gets$\texttt{list\_Ka[i]}
		\ForAll{\texttt{xb }$\in$ \texttt{list\_Xb}}
		\State \texttt{skip}$\gets$\texttt{False}
		\State \texttt{Ki}$\gets$\texttt{(popcount(b\&{}xb)\^{}ka)\&{}1}
		\State \texttt{syncthreads()}
		\ForAll{\texttt{g} $\in$ \texttt{list\_G}}
		\State compute \texttt{g}$(K)$ using the thread values \texttt{Ki}
		\If{\texttt{g}$(K)\leq 0$}
		\State \texttt{skip}$\gets$\texttt{True}
		\State break
		\EndIf
		\EndFor
		\If{\texttt{syncthreads\_or(skip)}}
		\State continue to the next \texttt{xb}
		\EndIf
		\State Reinitialize each value of \texttt{r} to \texttt{0} using the threads
		\If{\texttt{Ki}=\texttt{1}}
		\For{\texttt{k} $=1,\ldots,n$}
		\State increment \texttt{r[A[k][i]]} using the atomic add operation
		\If{\texttt{r[A[k][i]]} $\geq$ \texttt{3}}
		\State \texttt{skip}$\gets$\texttt{True}
		\State break
		\EndIf
		\EndFor
		\EndIf
		\If{\texttt{syncthreads\_or(skip)}}
		\State continue to the next \texttt{xb}
		\EndIf
		\State Add \texttt{K} to the list of results \texttt{list\_K}					
		\EndFor
		\EndFunction
		\Procedure{Main}{}
		\State Launch the $|\cX_a|$ blocks that correspond to all the pairs \texttt{(xa,Ka)} on the \textsc{Kernel}.
		\EndProcedure	
	\end{algorithmic}

\end{algorithm}

\begin{remark} When we say ``using the threads", we mean we evenly distribute the operations to perform among the threads.
	For example, to reinitialize the array \texttt{r}, we use the fact that we have $q$ threads that can set to zero $q$ coordinates simultaneously until all coordinates reset.
	Thus, it requires $\lceil \frac{N}{q}\rceil$ iterations, where $N$ is the number of ridges.
	We use a similar process for calculating the image by the affine functions $g\in\cG$.
\end{remark}

\begin{remark}
	We use the atomic add operation for incrementing values in \texttt{r} since many threads may write at the same memory location \texttt{r[k]}.
\end{remark}

The global complexity of this algorithm is
$$
    \cO\left(\frac{|\cX_a|}{p}\times|\cX_b|\times \frac{N}{q}\times(\alpha|\cG|+1)\right),
$$
where $\alpha$ is the average complexity of the atomic operation when called multiple times for a given $g\in\cG$.

\section{Preparation for applying the algorithm} \label{section:preparation}

\subsection{Finiteness of the problem and seedness}\label{subsection:finiteness}
Let $K$ and $L$ be simplicial complexes on the vertex sets $V$ and $W$, respectively, with $V\cap W=\emptyset$.
The \emph{join} of $K$ and $L$ is the simplicial complex $K \ast L = \{\sigma \cup \tau \mid \sigma \in K, \tau \in L \}$ on the vertex set $V\cup W$.
The \emph{link} of a face $\sigma$ in $K$ is the simplicial complex $Lk_K(\sigma)=\{ \tau \setminus \sigma \mid \sigma \subseteq \tau \in K \}$.
For the sake of simplicity, we denote the simplicial complex consisting of a single maximal simplex $\sigma$ by just $\sigma$.
The (simplicial) \emph{wedge} of $K$ at a vertex $v$ is $\wed_{v}(K)=(I*\Lk_{K}(v)) \cup (\partial I*K \setminus v)$, where $I$ is a $1$-simplex with two new vertices, and $K \setminus v$ is the simplicial complex on $V \setminus \{v\}$ consisting of the faces of $K$ that do not contain $v$.

A simplicial complex $K$ of dimension $n-1$ is a \emph{PL~sphere} if it has a subdivision that is isomorphic to any of  $\partial \Delta^{n-1}$.
It is a \emph{PL~manifold} if the link of each of its faces is a PL~sphere.
It is known that a PL~sphere is a PL~manifold, see \cite[Lemma~1.17]{hudson1969piecewise}.
We refer the reader to~\cite{Bagchi_Datta_1998} for more detailed definitions about PL~manifolds.

A PL~sphere $K$ is called a \emph{seed} if it is not a wedge of another PL~sphere $L$.
The following proposition follows immediately from the definition of the wedge.
\begin{proposition}[seedness from minimal non-faces]\label{proposition:seed no couple}
	A PL~sphere $K$ is a seed if and only if it satisfies the seedness condition; there is no face $\{v,w\}$ in $K$ such that for every minimal non-face $\sigma$ of $K$, we have either $\{v,w\}\subseteq \sigma$ or $v,w\notin \sigma$.
\end{proposition}
Remark that the seedness condition can be defined for general simplicial complexes.

Since the links of both new vertices in $\wed_{v}(K)$ are isomorphic to $K$, if $\wed_{v}(K)$ is a PL~sphere, so is $K$.
The converse also holds.

\begin{proposition}
	Let $K$ be a PL~sphere and $v$ a vertex of $K$.
	Then $\wed_{v}(K)$ is a PL~sphere.
\end{proposition}
\begin{proof}
	Suppose that $K$ is an $(n-1)$-dimensional PL~sphere.
	Since the suspension $K \ast \partial \{w_1, w_2\}$ of $K$ is isomorphic to an edge subdivision of $\wed_{v}(K)$, both have the same PL~structure.
	Moreover, $K \ast w_1$ is a subdivision of an $n$-simplex since $K$ is a PL~sphere. Hence $ K \ast \partial \{w_1, w_2\} = (K \ast w_1) \cup_K (K \ast w_2)$ is a subdivision of the boundary of an $n$-simplex.
\end{proof}

Let $K$ be a simplicial complex on the vertex set $[m]$.
A \emph{characteristic map} over $K$ is a map $\lambda \colon [m] \longrightarrow \Z^n$ satisfying the non-singularity condition for $K$: for each face $\sigma$ of $K$, $\lambda(\sigma)$ is a unimodular set.
Then, it is known that the existence of characteristic maps over $K$ is equivalent to the maximality of the Buchstaber number of $K$, i.e., $s(K) = m-n$, see \cite{Buchstaber-Panov2015}.
We call $K$ \emph{toric colorable} if it admits a characteristic map.

One can also consider its mod~$2$ analogue.
A mod~$2$ characteristic map over $K$ is a map $\lambda^\R \colon [m] \longrightarrow \Z_2^{n}$ satisfying that $\lambda^\R(\sigma)$ is a linearly independent set for all $\sigma \in K$.
Similarly, we call~$K$ \emph{$\Z_2^n$-colorable} if it admits a mod~$2$ characteristic map.

\begin{proposition}[\cite{ewald1996combinatorial}, \cite{Choi-Park2016}]
	Let $K$ be an $(n-1)$-dimensional PL~sphere and $v$ a vertex of $K$.
	Then $K$ is toric colorable if and only if so is $\wed_{v}(K)$, and $K$ is $\Z_2^n$-colorable if and only if $\wed_{v}(K)$ is $\Z_2^{n+1}$-colorable.
\end{proposition}

Notice that the composition of a characteristic map over $K$ and mod~$2$ reduction $\Z^n \to \Z_2^n$ yields a mod~$2$ characteristic map over $K$.
As a consequence, we firstly focus on $\Z_2^n$-colorable seeds.

We often see a mod~$2$ characteristic map $\lambda^\R$ as a \emph{characteristic matrix} $\begin{bmatrix} \lambda^\R(1) & \lambda^\R(2) & \cdots & \lambda^\R(m) \end{bmatrix}$.
Up to isomorphism, we may assume that the facet $\{1, 2, \ldots, n\}$ is in $K$.
With this assumption, to check $\Z_2^n$-colorability, it is enough to consider mod~$2$ characteristic maps of the form $\lambda^\R = \begin{bmatrix} I_n & M \end{bmatrix}$ since the non-singularity for $K$ is preserved by the left multiplication with an element of $\GL(n, \Z_2)$.

Let us define \emph{dual characteristic maps (DCM)} over $K$.
For $\lambda^\R = \begin{bmatrix} I_n & M \end{bmatrix}$, the DCM associated with $\lambda^\R$ is a map $\overline{\lambda^\R} \colon [m] \longrightarrow \Z_2^{m-n}$ such that $\overline{\lambda^\R} = \begin{bmatrix} \overline{\lambda^\R}(1) & \overline{\lambda^\R}(2) & \cdots & \overline{\lambda^\R}(m) \end{bmatrix}^t = \begin{bmatrix} M \\ I_{m-n} \end{bmatrix}$.
We shorten the term \emph{injective DCM} to \emph{IDCM}.

\begin{theorem} [\cite{CP_wedge_2}]\label{theorem:IDCM}
	Let $K$ be an $(n-1)$-dimensional PL~sphere with $m$ vertices and $v$, $w$ distinct vertices of $K$.
	Then the following statements are true.
	\begin{enumerate}
		\item If every facet of $K$ contains $v$ or $w$, then $K$ is a wedge or a suspension with respect to $v$ and $w$.
		\item If $K$ is a seed that is not a suspension, then every DCM over $K$ must be an IDCM.
	\end{enumerate}
	Statements~(2) and~(3) both imply:
	\begin{enumerate}[resume]
		\item If $K$ is a seed and $m-n \geq 3$, then $m \leq 2^{m-n}-1$.
	\end{enumerate}
\end{theorem}
We conclude from Statement~(3) of Theorem~\ref{theorem:IDCM} that there are only finitely many $\Z_2^n$-colorable seeds of fixed Picard number~$p$.
We now focus on the case $p=4$.
By Statement~(3) of Theorem~\ref{theorem:IDCM}, we have $n \leq 11$, implying that it is enough to enumerate colorable seeds of dimension up to $10$ ($n=11$).

\subsection{Checking isomorphism using minimal non-faces} \label{subsection:isom}

One demanding problem when enumerating simplicial complexes is dealing with isomorphism.
If a simplicial complex $K$ has $m$ vertices, then there are $m!$ possible relabeling for $K$.
Given two simplicial complexes $K$ and $L$, with respective vertex sets $V$ and $W$, we want to find if they are isomorphic.
One solution is to use McKay's graph isomorphism algorithm~\cite{McKay1981} on the face posets of $K$ and $L$.
We provide here a different solution for testing isomorphism by using their sets of minimal non-faces (MNF).

For every vertex $v$ of $K$, we define its \emph{color sequence} $c_K(v)$ by the increasing sizes of the minimal non-faces of $K$ containing $v$.
For example, if $\MNF(K)=\{\{1,2,3\},\{3,4\},\{4,5,6\},\{2,6\},\{1,6\}\}$, then $c_K(1) = (2,3)$, $c_K(5) = (3)$, and $c(6)=(2,2,3)$.
The color sequence of a vertex is preserved under isomorphism.
The procedure for checking the existence of an isomorphism between two simplicial complexes $K$ and $L$ is the following.
\begin{enumerate}
	\item Check whether $K$ and $L$ have the same combinatorial aspects such as the numbers of faces, and the numbers of non-faces.
	\item Check whether  $\{c_{K}(v)\colon v\in[m]\}=\{c_{L}(v)\colon v\in[m]\}$, by counting repetitions, using their MNF sets.
	\item We give partitions $V_1,\ldots,V_k$ and $W_1,\ldots,W_k$ of $V$ and $W$ with respect to color sequences in $K$ and $L$, respectively.
	We compute every relabeling $\phi_i\colon V_i\to W_i$ for every $i=1,\ldots,k$.
	They provide every relabeling $\phi = \phi_1\times\cdots\times \phi_k\colon V\to W$ that preserves the color sequences.
	If one $\phi$ sends one-to-one the minimal non-faces of $K$ to the ones of $L$ then $K$ is isomorphic to $L$.
\end{enumerate}
The number of relabeling that we compute is $(|V_1|!)\times\cdots\times(|V_k|!)$ instead of $|V|!$.
This provides a fine improvement when there are many different color sequences and only a few vertices share the same color sequence.

\subsection{Collecting PL spheres among weak pseudo-manifolds} \label{subsection:PL_sphere}
We need a criterion for a weak pseudo-manifold to be a PL~sphere.
We obtain this criterion in two steps.
First, when the Picard number is small enough, there is a nice characterization of PL~manifolds that are PL~spheres.

\begin{theorem}[\cite{Bagchi2005}]
	\label{theorem:combisp}
	Let $K$ be a PL~manifold such that $\Pic(K) \leq 7$.
	If $K$ is a $\Z_2$-homology sphere, then it is a PL~sphere.
\end{theorem}

By using the above theorem and by the definition of PL~manifolds, we obtain the following lemma.

\begin{lemma}[PL~sphereness]
	\label{lemma:PL_sphereness}
	A weak pseudo-manifold $K$ of Picard number~$\leq 7$  is a PL~sphere if and only if the link of any face (including the empty face) of $K$ is a $\Z_2$-homology sphere.
\end{lemma}
\begin{proof}
	The ``only if'' part is immediate, so it is enough to show the ``if'' part.
	Suppose that the link of any face of $K$ is a $\Z_2$-homology sphere.
	By applying Theorem~\ref{theorem:combisp}, let us prove that the link of each face of $K$ is a PL~sphere.
	
	We use induction on the dimension of the link of a face.
	We remark that the link of each $(n-2)$-face of $K$ is the $0$-sphere $S^0$ by the definition of weak pseudo-manifolds.
	In particular, it is a PL~sphere.
	For $k \leq n-3$, let $\sigma$ be a $k$-face of $K$, and $L = \Lk_K(\sigma)$ its link.
	Note that $\Lk_L(v)$ for a vertex $v$ of $L$ is equal $\Lk_K( \{v\} \cup \sigma )$.
	Therefore, if the link of any $(k+1)$-face of $K$ is a PL~sphere, then $L$ is a PL~manifold.
	Since $L$ is a $\Z_2$-homology sphere by assumption and $\Pic L \leq \Pic K \leq 7$, and $L$ is a PL~sphere by Theorem~\ref{theorem:combisp}.
	By induction, the link of each face is a PL~sphere.
\end{proof}

If we proceed our enumeration inductively, we can use the following method for verifying the PL~sphereness of a given weak pseudo-manifold.
Let $\cS^\circ(n,p)$ denote the set of $\Z_2^n$-colorable seeds of Picard number~$p$ and dimension $n-1$, up to isomorphism.
Let us suppose that we have obtained all $\cS^\circ(k,p)$ for $k<n$ and $p\leq 4$.
Given any $\Z_2^n$-colorable weak pseudo-manifold $K$, we apply the following procedure to check if it is a PL~sphere:
\begin{enumerate}
	\item Check if the $\Z_2$-Betti numbers of $K$ are the ones of a sphere, namely $(1,0,\ldots,0,1)$.
	\item For every vertex $v$ of $K$, let $K_v = \Lk_K(v)$, and let $L_v$ be the seed that $K_v$ is obtained from.
	Since the PL~sphereness property is invariant under the wedge operation, we need to check for every $v$ that $L_v$ is isomorphic to a representative in $\cS^\circ(k,p)$, for some $p\leq 4$ and $k<n$.
	For this purpose, we use the isomorphism-checking method we provided in Section~\ref{subsection:isom}.
\end{enumerate}

We now have the tools for checking:
\begin{itemize}
	\item the seedness condition on a simplicial complex with Proposition~\ref{proposition:seed no couple},
	\item the existence of an isomorphism between two simplicial complexes in Section~\ref{subsection:isom}, and
	\item the PL~sphereness of a weak pseudo-manifold of Picard number~$4$ with Lemma~\ref{lemma:PL_sphereness}.
\end{itemize}

\section{Toric colorable PL spheres of Picard number four} \label{section:Pic4}
In this section, we focus on enumerating all $(n-1)$-dimensional toric colorable seeds of Picard number~$4$.

\subsection{A first intuitive procedure}

%
%
%

One could intuitively try to find all PL~spheres,
and compute their (real) Buchstaber numbers.
However, it is hopeless when we consider high dimensions.
We could obtain results up to $n=6$ by applying either Algorithm~\ref{algorithm:main} or other known methods such as lexicographic enumeration \cite{Sulanke-Lutz2009}, but it seems to take too long to finish for bigger $n$.
\begin{remark}\label{theorem:buchstaber_number}
	Up to isomorphism, one can compute the numbers of $(n-1)$-dimensional PL~spheres and seeds of Picard number~$4$ up to $n=6$, and their real Buchstaber numbers $s^\R$ as follows.
\begin{center}
		\begin{tabular}{lrrrrrr}
		\toprule
		$n$   &        & 2 & 3  & 4  & 5  & 6  \tabularnewline \midrule
		\multicolumn{2}{l}{PL~spheres}   & 1  & 5  & 39 & 337 & 6257 \tabularnewline  \rule{0pt}{3ex}
		&$s^\R=4$           & 1  & 5  & 37 & 281 & 2353 \tabularnewline
		&$s^\R=3$           & 0  & 0  & 2  & 56  & 3904 \tabularnewline \midrule
		\multicolumn{2}{l}{seeds} & 1  & 4  & 23 & 194 & 4237 \tabularnewline \rule{0pt}{3ex}
		&$s^\R=4$           & 1  & 4  & 21 & 142 & 733  \tabularnewline
		&$s^\R=3$           & 0  & 0  & 2  & 52  & 3504 \tabularnewline
		\bottomrule
	\end{tabular}
\end{center}
\end{remark}

Since we focus on $\Z_2^n$-colorable seeds, we use a different approach in all that follows.

\subsection{Enumeration for $n \leq 10$}\label{method:kernel}
In this subsection, we enumerate all $(n-1)$-dimensional $\Z_2^n$-colorable seeds on $[m]$ of Picard number~$4$ for $n\leq 10$.


Suppose that a $\Z_2^n$-colorable seed supports an IDCM.
We first investigate the combinatorial structure of the IDCM itself.

A \emph{matroid} $M$ is a simplicial complex with the \emph{augmentation property}; for any $\tau, \sigma \in M$ with $|\tau| < |\sigma|$, there exists $x \in \sigma \setminus \tau$ such that $\tau \cup \{x\} \in M$.
The \emph{dual matroid} $\overline{M}$ of $M$ is defined on the same vertex set as $M$, and its facets are the complements of each facets of $M$, which are called the \emph{cofacets} of $M$.
For a full row-rank $n \times m$ matrix $\lambda^\R$ over $\Z_2$, the simplicial complex $M_{\lambda^\R}$, whose facets are the sets of column indexes of $n$ linearly independent columns of $\lambda^\R$, forms a matroid.
This matroid is called the \emph{binary matroid} associated with $\lambda^\R$.
Therefore, $K$ supports a mod~$2$ characteristic map $\lambda^\R$ if and only if $K$ is a subcomplex of $M_{\lambda^\R}$.
By linear Gale duality \cite{ewald1996combinatorial}, the dual matroid $\overline{M_{\lambda^\R}}$ is equal to $M_{\overline{\lambda^\R}^t}$.
We can easily verify the following proposition using the definitions of $M_{\lambda^\R}$ and $M_{\overline{\lambda^\R}^t}$.

\begin{proposition}\label{prop:linear matroid}
	Let $K$ be an $(n-1)$-dimensional simplicial complex on $[m]$ and $\overline{\lambda^\R}$ an $m \times (m-n)$ matrix over $\Z_2$ of rank $m-n$.
	Then $K$ supports $\overline{\lambda^\R}$ as a DCM if and only if it is a subcomplex of $\overline{M_{\overline{\lambda^\R}^t}} =M_{\lambda^\R}$.
\end{proposition}

As we recall, reducing the number of facets in the input of Algorithm~\ref{algorithm:main} leads to a smaller dimension of the mod~$2$ kernel of the ridge-facet incidence matrix, resulting in faster execution of Algorithm~\ref{algorithm:main}.
Proposition~\ref{prop:linear matroid} provides us with a smaller set of facets, as desired.

Furthermore, we leverage the upper bound theorem for facets of PL~spheres, see \cite{Stanley1996book}.
According to this theorem, the number of facets of an $(n-1)$-dimensional simplicial sphere of Picard number~$4$ is less than or equal to the number of facets of the cyclic $n$-polytope $C^n(n+4)$ with $n+4$ vertices.
It is known that the number of facets of $C^n(n+4)$ is
$$
    f_{n-1} (C^n(n+4)) = \binom{n+4- \lceil \frac{n}{2} \rceil}{4} + \binom{n+3- \lfloor \frac{n}{2} \rfloor}{4},
$$ see \cite{Buchstaber-Panov2015} for example.

This condition is represented by the affine function $g(K)=f_{n-1}(C^n(n+4))-\lVert K \rVert_1 +1$, where $\lVert K \rVert_1$ is the $1$-norm of the vector $K$, which corresponds to the number of facets of $K$.
Let $\overline{\lambda^\R} \colon [m] \longrightarrow \Z_2^4$ be an injective map, and denote by $\cF(\lambda^\R)=\cF(M_{\lambda^\R})$ the set of facets of the associated binary matroid.
Algorithm~\ref{algorithm:main}, with inputs $\cF(\lambda^\R)$ and the affine function $g$, outputs the set of all weak pseudo-manifolds that support $\overline{\lambda^\R}$ and satisfy the upper bound theorem.

At first glance, it might seem necessary to run the algorithm on each of the $\binom{11}{n}\times n!$ injective maps $\overline{\lambda^\R}$, even if we fix $\begin{bmatrix} \overline{\lambda^\R}(n+1) & \overline{\lambda^\R}(n+2) & \overline{\lambda^\R}(n+3) & \overline{\lambda^\R}(n+4) \end{bmatrix} = I_4$.
However, we can significantly reduce this large number of cases by observing that many injective maps yield the same outputs up to isomorphism.

Let $\Lambda(n, p)$ be the set of all $ (n+p) \times p$ matrices over $\Z_2$ of the form $\begin{bmatrix} M \\ I_p \end{bmatrix}$, with no repeated rows.
We consider the product of two symmetric groups $\mathfrak{S}_n \times \mathfrak{S}_p$
which acts on $\Lambda(n, p)$ as follows: $\left(\begin{bmatrix} M \\ I_p \end{bmatrix}, (s, t) \right) \mapsto \begin{bmatrix} P_s^t M P_t \\ I_p \end{bmatrix}$, where $P_s$ and $P_t$ are column permutation matrices corresponding to permutations $s$ and $t$.
Let us call each element of $\Lambda(n, p) / \mathfrak{S}_n \times \mathfrak{S}_p$ an \emph{IDCM orbit}.

\begin{proposition}\label{prop:IDCM orbit}
	For $(s, t) \in \mathfrak{S}_n \times \mathfrak{S}_p$, there is an isomorphism between the binary matroids associated to $\overline{\lambda^\R} \in \Lambda(n, p)$ and $\overline{\lambda^\R} \circ (s, t) \in \Lambda(n, p)$.
\end{proposition}
\begin{proof}
	We partition the vertex set $[m]$ into $V_1=[n]$ and $V_2=\{n+1, n+2, \ldots, n+p\}$.
	It is evident that the matrix $\begin{bmatrix} P_s^t M \\ I_p \end{bmatrix}$ preserves the same non-singularity information on cofacets, with the vertices in $V_1$ relabeled according to permutation $s$.
	Additionally, since $P_t$ is invertible, we have $M_{(\overline{\lambda^\R}P_t)^t} = M_{\overline{\lambda^\R}^t}$.
	Then, applying $t$ to $V_2$ in the equation
	\begin{align*}
		\begin{bmatrix} P_s^t M P_t \\ I_p \end{bmatrix} = \begin{bmatrix} P_s^t M \\ P_t^t I_p \end{bmatrix} P_t,
	\end{align*}
 yields the same non-singularity conditions for $\begin{bmatrix} P_s^t M \\ I_p \end{bmatrix}$ and $\begin{bmatrix} P_s^t M P_t \\ I_p \end{bmatrix}$.
\end{proof}

Let $\Lambda^\circ(n,4)\subseteq \Lambda(n,4)$ be a set containing one representative per IDCM orbit.
By Proposition~\ref{prop:IDCM orbit}, it is enough to input Algorithm~\ref{algorithm:main} with $\cF(\lambda^\R)$ for all $\lambda^\R\in\Lambda^\circ(n,4)$.
Table~\ref{table:table_IDCM} displays how efficient our method is.
It provides the dimension of the kernel of  the ridge-facet incidence matrix $A\left(\binom{[m]}{n}\right)$ and the size of the set $\cX\left(\binom{[m]}{n}\right)$ which represents the number of element in its kernel for which we should verify whether they are weak pseudo-manifolds, together with $\max_{\overline{\lambda^\R}}(\dim\ker A(\cF(\lambda^\R)))$ and $\max_{\overline{\lambda^\R}} |\cX(\cF(\lambda^\R))|$.
The number of IDCM orbits of $\Lambda(n,4)$ and the computation time of the call of Algorithm~\ref{algorithm:main} at line~\ref{line:5} of Algorithm~\ref{algorithm:colorable_seeds} are also provided.
This demonstrates that our reductions enable computability of the problem, for example with $n=10$:
\begin{itemize}
	\item the choice of the convenient basis made in Section~\ref{subsection:enum_weak_psdmfd} reduces the number of cases from $2^{286}\simeq 1e86$ to $5e74$ for the set of facets $\binom{[m]}{n}$, and from $2^{56}\simeq 7e16$ to $4e14$ for  $\cF(\lambda^\R)$, and
	\item taking into account that $\Z_2^n$-colorable seeds are subcomplexes of the binary matroid associated to an IDCM divides the number of cases to compute by a factor of $10^{60}$.
\end{itemize}

\begin{table}[ht]
	\begin{tabular}{lcccccccccc}
		\toprule
		$n$                                                                     &  2  &  3   &  4   &  5   &  6   &  7  &  8   &  9   &   10   &     11      \tabularnewline \midrule
		$\dim \ker A\left(\binom{[m]}{n}\right)$                                & 10  &   20   &  35    &   56   &   84   &  120   &    165  &   220   &   286     &     364        \tabularnewline
		$|\cX \left(\binom{[m]}{n}\right)|$                                     &  352   &  2e5    &  1e9    &   3e14   &  7e18    &  8e21   &  3e31    &    4e57  &   5e74     &    2e93         \tabularnewline \midrule
		Number of IDCM orbits                                                   &  7  &  16  &  28  &  35  &  35  & 28  &  16  &  7   &   3    &      1      \tabularnewline
		$\max_{\overline{\lambda^\R}}(\dim\ker A(\cF(\lambda^\R)))$ &  7  &  13  &  21  &  24  &  28  & 34  &  42  &  48  &   56   &     64      \tabularnewline
		$\max_{\overline{\lambda^\R}} |\cX(\cF(\lambda^\R))|$ & 56  & 3e3  & 5e5  & 1e6  & 2e7  & 9e8 & 1e11 & 3e12 & 4e14 &   4e16    \tabularnewline \midrule
		Time spent for one orbit                                                & 1ms & 10ms & 0.1s & 0.6s & 1.3s & 3m  & 15m  &  2h  &  12d   & \textbf{3y} \tabularnewline \bottomrule
	\end{tabular}
	\vspace*{0.2cm}
	\caption{Data table for Picard number~$4$ and $n=2,\ldots,11$. The time spent refers to Algorithm~\ref{algorithm:main} running on an Nvidia Quadro A5000. The time written in bold in the case $n=11$ is an estimation.} \label{table:table_IDCM}
\end{table}
\begin{algorithm}
	The full procedure for obtaining every $(n-1)$-dimensional seed of Picard number~$4$ supporting an IDCM is as follows.\label{algorithm:colorable_seeds}
	\begin{algorithmic}[1]
		\algrenewcommand\algorithmicrequire{\textbf{Input:}}
		\algrenewcommand\algorithmicensure{\textbf{Output:}}
		\algrenewcommand\algorithmicprocedure{\textbf{Procedure}}
		\algrenewcommand\algorithmicfunction{\textbf{Function}}
		\Require{Integer $n\geq 2$.}
		\Ensure{The set $\cK^\circ(n,4)$ of $(n-1)$-dimensional seeds of Picard number~$4$ supporting an IDCM up to isomorphism.}
		\Procedure{GetIDCM-ColorableSeedsPic4}{$n$}
		\State Compute $\Lambda^\circ(n,4)\subseteq \Lambda(n,4)$ a set containing one representative for each IDCM orbit.
		\State $\cK(n,4)\gets\emptyset$
		\ForAll{$\lambda^\R \in \Lambda^\circ(n,4)$}
		\State $\cK(\lambda^\R) \gets $ output of Algorithm~\ref{algorithm:main} with inputs $\cF(\lambda^\R)$ and $\cG=\{g\}$, with $g(K) = f_{n-1}(C(n,n+4)) -\|K\|_1 +1$\label{line:5}
		\State $\cK(n,4)\gets \cK(n,4)\cup \cK(\lambda^\R)$
		\EndFor \label{line:7}
		\ForAll{$K\in\cK(n,4)$}
		\If{$\Pic(K)<4$} discard $K$ \EndIf
		\If{$K$ does not satisfy the seedness condition} discard $K$ \EndIf\label{line:12}
		\EndFor\label{line:13}
		\State Select $\cK^\circ(n,4)\subseteq \cK(n,4)$ with one representative $K$ up to isomorphism
		\ForAll{$K\in\cK^\circ(n,4)$}\label{line:15}
		\If{$K$ is not a PL~sphere} discard $K$ \EndIf
		\EndFor
		\EndProcedure\label{line:19}
	\end{algorithmic}
\end{algorithm}
Running  Algorithm~\ref{algorithm:colorable_seeds} for $n \leq 10$ gives Table~\ref{table:output}.
\begin{table}[ht]
	\begin{tabular}{l*{10}{c}}
		\toprule
		$n$ & $2$ & $3$  & $4$   & $5$   & $6$    & $7$    & $8$    & $9$     & $10$  \tabularnewline \midrule
		$|\cK(n,4)|$ at line~\ref{line:7}& $90$  & $1119$ & $20383$ & $79877$ & $322837$ & $503624$ & $469445$ & $224854$ & $99374$ \tabularnewline
		$|\cK(n,4)|$ at line~\ref{line:13} & $22$  & $578$  & $13679$ & $47012$ & $204714$ & $310217$ & $305280$ & $140933$  & $57956$ \tabularnewline
		$|\cK^\circ(n,4)|$ at line~\ref{line:15} & $2$  & $5$  & $49$ & $256$ & $1791$ & $2194$ & $1401$ & $381$  & $56$ \tabularnewline
		$|\cK^\circ(n,4)|$ at line~\ref{line:19}  & $1$ & $4$  & $20$  & $142$ & $733$  & $1190$ & $776$  & $243$   & $39$  \tabularnewline \bottomrule
	\end{tabular}
	\vspace*{0.2cm}
	\caption{The output of Algorithm~\ref{algorithm:colorable_seeds} for $n\leq 10$.\label{table:output}}
\end{table}

Now, all that remains are the seeds that do not admit any IDCM.
We show there actually remains a single one not outputted by Algorithm~\ref{algorithm:colorable_seeds}.

From Theorem~\ref{theorem:IDCM}, they must be suspensions.
Let $L=\partial{[v, w]} \ast K$ be the suspension of an $(n-2)$-dimensional simplicial complex $K$, and suppose that $L$ is $\Z_2^{n}$-colorable.
We may assume that a characteristic map $\lambda^\R$ over $L$ satisfies $\lambda^\R(v)=\begin{bmatrix} 1 & 0 & \cdots & 0  \end{bmatrix}^t$.
Then for any facet $\{v\} \cup \{v_1, \ldots, v_{n-1} \}$ of $L$, the $(1, v)$ minor of the matrix $\begin{bmatrix} \lambda^\R(v) & \lambda^\R(v_1) & \cdots & \lambda^\R(v_{n-1}) \end{bmatrix}$ is equal to $1$.
This implies that $\Lk_L(1)=K$ is $\Z_2^{n-1}$-colorable.
Hence, the suspension operation preserves $\Z_2^{n}$-colorability while also preserving seedness, and increases the Picard number by one.
Therefore, it is sufficient to consider the suspensions of the seeds of Picard number~$3$.

The three $\Z_2^n$-colorable seeds of Picard number~$3$ are the boundaries of a pentagon, a $3$~dimensional cross polytope, and a cyclic polytope $C^4(7)$ \cite{Erokhovets2011}.

The suspension of a pentagon and a cyclic polytope both support an IDCM and 
have been obtained in Table~\ref{table:output} for $n=3$ and $5$.
Finally, the boundary of a cross polytope does not support any IDCM, but does support a DCM, so we add it to the result $\cK^\circ(4,4)$ in Table~\ref{table:output}.

\begin{theorem}
	Up to isomorphism, the number of $\Z_2^n$-colorable seeds of dimension $n-1$ and Picard number~$4$ for $n\leq 10$ is as follows.
	\begin{center}
		\begin{tabular}{l*{10}{c}}
			\toprule
			$ n$ & $2$ & $3$ & $4$  &  $5$  &  $6$  &  $7$   &  $8$  &  $9$  & $10$  \tabularnewline \midrule
			$\Z_2^n$-colorable seeds                 & $1$ & $4$ & $20+1$ & $142$ & $733$ & $1190$ & $776$ & $243$ & $39$  \tabularnewline \bottomrule
		\end{tabular}
	\end{center}
\end{theorem}

\subsection{Enumeration for $n=11$} \label{subsection:n=11}
As shown in Table~\ref{table:table_IDCM}, the time complexity of the extreme case $n=11$ remains too long.
To address this, we leverage the results obtained from the dimension just below to construct the seeds for this extreme case.

Let $K$ be a $\Z_2^{11}$-colorable seed on $\{1, 2, \ldots, 15 \}$ of dimension $10$ ($n=11$).
We know that the link of the vertex $15$ has Picard number~$\leq 4$ and is a $\Z_2^{10}$-colorable seed, which we have already enumerated.
We construct all $\Z_2^{11}$-colorable seeds of dimension $10$ from the $\Z_2^{10}$-colorable ones of dimension $9$.
Firstly, if $K$ has only vertices whose links have Picard numbers at most 2, then $K$ is the boundary of a product of simplices \cite{Grunbaum2003}, and therefore not a seed.
Suppose that the link of $15$ has Picard number~$3$.
Since there is no $9$-dimensional seed of Picard number~$3$, it implies that the link of vertex $15$ is not a seed.
By the following lemma, we can identify another vertex of $K$ whose link has Picard number~$4$.
\begin{lemma}
	Let $K$ be a seed of Picard number~$4$.
	Assume that $K$ has a vertex $v$ such that $\Pic(\Lk_K(v)) = 3$ and there exist two vertices $v_1$ and $v_2$ of $\Lk_K(v)$ such that every facets of $\Lk_K(v)$ contains either $v_1$ or $v_2$.
	Then, there is a vertex of $K$ whose link has Picard number~$4$.
\end{lemma}
\begin{proof}
	Let $\{v_1\} \cup \sigma$ be a facet without $v_2$.
	There is one more facet containing $\sigma$ since it is a ridge of the PL~sphere $\Lk_K(v)$.
	By the assumption, it must be $\{v_2\} \cup \sigma$.
	That shows every vertex in $\Lk_K(v)$ forms an edge with both $v_1$ and $v_2$.
	
	Let $w$ be the vertex not in $\Lk_K(v)$.
	If $v_1 \in \Lk_K(w)$, then $\Lk_K(v_1)$ has Picard number~$4$.
	If both $v_1, v_2 \not\in \Lk_K(w)$, then $\Lk_K(w)$ is an $(n-2)$-dimensional PL~sphere with $n$ vertices.
    This means that $\Lk_K(w) = \partial\Delta_{n-1}$.
    Then if $w^\prime$ is a vertex of $\Lk_K(w)$ other than $w$, then $\Pic(\Lk_K(w^\prime)) = 4$.
\end{proof}
That implies any $\Z_2^{11}$-colorable seed of Picard number~$4$ has a vertex whose link also has Picard number~$4$, which we relabel as vertex $15$.
Before we apply Algorithm~\ref{algorithm:colorable_seeds} for this case, we need some preparation as follows.
We firstly select an injective map $\bar{\mu}\colon \{1,\ldots, 14\} \to \Z_2^4$ and choose a $9$-dimensional PL~sphere $L$ that supports $\bar{\mu}$.
We see $L$ as the link of the vertex $15$ in some $\Z_2^{11}$-colorable seed $K$ supporting some IDCM $\overline{\lambda^\R}$ with the restriction $\overline{\lambda^\R} \vert_{\{1,\ldots, 14\}} = \bar{\mu}$.
Since $\left| \Z_2^4  \setminus \{0\}\right| = 15$, once $\bar{\mu}$ is chosen, $\overline{\lambda^\R}$ is uniquely determined.
There are $114$ $9$-dimensional $\Z_2^{10}$-colorable PL~spheres that support an IDCM, among which $39$ are seeds and $75$ are non-seeds. They can be obtained from Algorithm~\ref{algorithm:colorable_seeds} by skipping the step that discards the non-seeds, for instance.

Let $\hat L$ be the simplicial complex $\{ \sigma \cup \{15\} \mid \sigma \in L \}$.
All PL~spheres $K$ having its vertex $15$ whose link is $L$ contains $\hat{L}$.
That provides the following conditions on the components of $K\in\Z^M$:
\begin{enumerate}
	\item for all $\hat{L}_j=1$, $K_j=1$, and
	\item for all $\hat{L}_j = 0$ with $\hat{L}_j \ni \{15\}$, $K_j=0$.
\end{enumerate}
We will denote by $I$ and $J$  the set of indexes of the facets satisfying Condition~(1), respectively Condition~(2).
After reordering the rows of $B$, the two conditions appear as follows
\begin{align}\label{equation:condition_X}
	BX = \begin{bmatrix} B_I \\ B_J \\ B_{[M] \setminus (I \cup J)} \end{bmatrix}X = \begin{bmatrix} \mathbf{1} \\ \mathbf{0} \\ \star \end{bmatrix}.
\end{align}
A mod~$2$ Gaussian elimination process on the columns of $B$ gives a column-reduced echelon form $\widetilde{B}$ which yields another set of generators for the mod~$2$ kernel of $A$.
Denote by $s_I$ and $s_J$ the maximal index of non-zero column of $\widetilde{B}_{I}$ and of $\widetilde{B}_J$, respectively.
To respect conditions (1) and (2), we need $$x_t =\begin{cases} 1&t=1,\ldots,s_I\\0&t=s_I+1,\ldots, s_J\\ \star& \text{otherwise}\end{cases},$$
for $X = (x_1, \ldots, x_{M})^t$.
If no such $X$ satisfy this conditions, then there is no $\Z_2^{11}$-colorable seed $K$ that supports $\overline{\lambda^\R}$ and whose link of the vertex $15$ is $L$.

Applying Algorithm~\ref{algorithm:colorable_seeds} with a revised version of the initialization of Algorithm~\ref{algorithm:main} that takes into account Condition~\eqref{equation:condition_X} on the entries of $X$ yields the following.

\begin{theorem}
	There are exactly $4$ $\Z_2^{11}$-colorable $10$-dimensional seeds of Picard number~$4$.
\end{theorem}

\subsection{Toric colorability} \label{subsectoin:toric_colorable}

Remember that it is enough to check whether each seed among the $\Z_2^n$-colorable ones supports a characteristic map for obtaining all the toric colorable seeds of Picard number~$4$.

Let $K$ be a $\Z_2^n$-colorable seed of Picard number~$4$, and $\lambda^\R$ a mod~$2$ characteristic map over $K$.
If we regard each image vector of $\lambda^\R$ as a $\{0,1\}$-vector of $\Z^n$, and denote $\lambda$ the obtained map, then $\lambda$ is not necessarily a characteristic map over $K$.
Therefore, we change some $1$'s in the image vectors of $\lambda$ to $-1$'s until it becomes a characteristic map over $K$.
Brute-forcing this method provides at least one characteristic map supported by every $\Z_2^n$-colorable seed we enumerated.
The toric colorability is thus equivalent to the $\Z_2^n$-colorability for PL~spheres of Picard number~$4$.
This yields the full theorem.
\main

\section{Application to the normalized space of rational curves on toric manifolds of Picard number four} \label{sec:min_comp}
This section is devoted to answering a question of Chen, Fu, and Hwang in~\cite{fu2014minimal} and assume the reader is familiar with it.
We also refer to~\cite{Kollar_1995,Hwang_2001} for more details about rational curves on algebraic varieties.
Let $X$ be a toric manifold.
For an irreducible component $\cK$ of the normalized space of rational curves on $X$, denote by $\rho \colon \cU \longrightarrow \cK$ and $\mu \colon \cU \longrightarrow X$ the associated universal family morphisms.
The irreducible component $\cK$ is called a \emph{minimal component} if $\mu$ is dominant and for a general point $x \in X$, the variety $\mu^{-1}(x)$ is complete.
Members of such $\cK$ are called $\emph{minimal rational curves}$ and the \emph{degree} of $\cK$ is defined by the degree of the intersection of the anti-canonical divisor of $X$ with any member in $\cK$.

Recall that $X$ is characterized by a complete non-singular fan with its underlying simplicial complex $K$ on $[m]$ and its primitive ray vectors $\lambda(i)$, for $i \in [m]$.
In particular, $\lambda$ is a characteristic map over $K$.
Conversely, if a characteristic map $\lambda$ over $K$ gives a fan with its underlying simplicial complex $K$, then $\lambda$ is called \emph{fan-giving}.

For each minimal non-face $\{v_1, v_2, \ldots, v_k\}$ of $K$, the set $\{\lambda(v_1), \lambda(v_2), \ldots, \lambda(v_k)\}$ is called a \emph{primitive collection} of $(K, \lambda)$.

\begin{theorem}{{\cite[Proposition~3.2]{fu2014minimal}}}
	Let $X$ be a toric manifold of complex dimension~$n$, and let $(K,\lambda)$ represent its associated fan.
    The minimal components of degree $k$ on $X$ bijectively correspond to primitive collections $\{\lambda(v_1), \lambda(v_2), \ldots, \lambda(v_k)\}$ of $(K, \lambda)$ such that $\lambda(v_1) + \dots + \lambda(v_k) = 0$.
\end{theorem}

We consider two primitive collections $V = \{\lambda(v_1), \lambda(v_2), \ldots, \lambda(v_k)\}$ and $W = \{\lambda(w_1),\ldots, \lambda(w_l)\}$ that correspond to two minimal components.
Assume that they intersect, so without loss of generality, $\lambda(v_k) = \lambda(w_l)$.
Then \begin{align*}
	\lambda(v_1) + \dots + \lambda(v_{k-1}) &= -\lambda(v_k)
	= -\lambda(w_l)
	= \lambda(w_1) + \dots + \lambda(w_{l-2}).
\end{align*}
This means that the two cones generated by $V \setminus \{v_k \}$ and $W \setminus \{w_l \}$ are the same.
Hence for two primitive collections $V$ and $W$ corresponding to minimal components, there are only two possibilities: either $V=W$, or $V \cap W = \varnothing$.
Using this property and the previous theorem one obtains the following inequality.
\begin{proposition}{{\cite[Proposition~3.5]{fu2014minimal}}}\label{proposition:inequality}
	Let $X$ be a toric manifold of complex dimension~$n$ and Picard number~$p$.
	Then \begin{align*}
		\sum_{k=0}^{n-1} n_k(k+2) \leq n+p,
	\end{align*}
	where $n_k$ is the number of minimal components in the normalized space of rational curves on $X$ of degree $k+2$.
\end{proposition}

We consider the fan associated to $X$, represented by a pair $(K, \lambda)$.
Through a direct interpretation of the inequality, the equality holds if and only if there is a partition $P$ of the vertex set of $K$ such that for each $\sigma = \{v_1, \ldots, v_k\} \in P$, $\sigma$ is a minimal non-face of $K$, and $\lambda(v_1) + \dots + \lambda(v_k) = 0$.
Let us call such partition \emph{optimal}.
The rest of this section is devoted to finding in which cases there exists such an optimal partition for $K$ of Picard number~$\leq 4$ that supports a fan-giving characteristic map.

Firstly, observe that the left multiplication of a characteristic matrix by an invertible matrix does not affect whether it has an optimal partition.
Such two matrices are called Davis-Januszkiewicz equivalent (or simply \emph{D-J equivalent}).
Hence, we suppose that the first $n$ columns of an $n \times m$ characteristic matrix over $K$ form the $n \times n$ identity matrix with the assumption that $\{1, 2, \ldots, n\}$ is a facet of $K$.

For a simplicial complex $K$, assume that the following two matrices are (mod~$2$) characteristic maps over $K$;
\begin{align} \label{equation: characteristic maps} \begin{split} \lambda &=\begin{bmatrix} 1 & \mathds{O} & \textbf{a} \\
			\mathds{O} & I_{n-1} & A
		\end{bmatrix}\\
		\mu &=\begin{bmatrix} 1 & \mathds{O} & \textbf{b} \\
			\mathds{O} & I_{n-1} & A
		\end{bmatrix}.
	\end{split}
\end{align}
Then, using the notations introduced in \cite{Choi-Vallee2022}, the matrix \begin{align} \label{equation: wedged characteristic map}
	\lambda \wedge_1 \mu = \begin{bmatrix}
		1 & 0 & \mathds{O} & \textbf{a} \\
		0 & 1 & \mathds{O} & \textbf{b} \\
		\mathds{O} & \mathds{O} & I_{n-1} & A
	\end{bmatrix}
\end{align}
is a (mod~$2$) characteristic map over $\wed_1(K)$ if it satisfies the (mod~$2$) non-singularity condition.
For any other vertex $v$, one can construct a (mod~$2$) characteristic map over $\wed_v(K)$ from two (mod~$2$) characteristic maps over $K$ similarly.

\begin{theorem}{\cite{Choi-Park2016}} \label{theorem: wedge fan-giving}
	For a vertex $v$ of a simplicial complex $K$, every (mod~$2$) characteristic map over $\wed_v(K)$ is of the form $\lambda \wedge_v \mu$ for two (mod~$2$) characteristic maps $\lambda$ and $\mu$ over $K$ up to D-J equivalence.
	Moreover, $\lambda \wedge_v \mu$ is fan-giving if and only if both $\lambda$ and $\mu$ are fan-giving.
\end{theorem}
\begin{remark}
	Let $v$ be a vertex of $K$.
	Notice that by taking any two (mod~$2$) characteristic maps $\lambda$ and $\mu$ over $K$ we cannot always construct the matrix ``$\lambda\wedge_v \mu$".
	However, for a single (mod~$2$) characteristic map over $K$, the matrix $\lambda \wedge_v \lambda$ can always be constructed.
	Such (mod~$2$) characteristic map over $\wed_v(K)$ is called the \emph{canonical extension} of $\lambda$ at $v$.
	For more details on the compatibility of the operation $\wedge_v$ in the mod~$2$ case, see~\cite{Choi-Vallee2022}.
\end{remark}

\begin{lemma} \label{lemma: wedge optimal partition}
	For a simplicial complex $K$ and a vertex $v$ of $K$, let $\Lambda = \lambda \wedge_v \mu$ be a fan-giving characteristic map over $\wed_v(K)$, where $\lambda$ and $\mu$ are characteristic maps over $K$.
	Then $(\wed_v(K), \Lambda)$ has an optimal partition if and only if both $(K, \lambda)$ and $(K, \mu)$ have optimal partitions.
\end{lemma}

\begin{proof}
	Without loss of generality, it is enough to consider the case $v=1$.
	Then we can additionally assume that characteristic maps are of the forms~\eqref{equation: characteristic maps} and~\eqref{equation: wedged characteristic map}.
	Then for a subset $\{v_1, \ldots, v_k\}$ of the vertex set of $K$, $\Lambda(v_1) + \dots + \Lambda(v_k) = 0$ if and only if the sum is zero component-wise.
	Hence $\Lambda(v_1) + \dots + \Lambda(v_k) = 0$ if and only if $\lambda(v_1) + \dots + \lambda(v_k) = 0$ and $\mu(v_1) + \dots + \mu(v_k) = 0$.
\end{proof}

By Theorem~\ref{theorem: wedge fan-giving} and Lemma~\ref{lemma: wedge optimal partition}, it is enough to investigate only seeds of Picard number~$p$ to determine which fan of Picard number~$p$ has an optimal partition.
Then, by Theorem~\ref{thm:main}, we obtain the following corollary.
\begin{corollary}\label{corollary:min_comp}
	A PL~sphere $K$ of Picard number~$p \leq 4$ supports a fan-giving characteristic map $\lambda$ that has an optimal partition if and only if $K$ is achieved by a sequence of wedge operations from the boundary of
	\begin{itemize}
		\item a $1$-simplex if $p=1$,
		\item a square if $p=2$,
		\item a $3$-dimensional cross polytope if $p=3$, and
		\item either a hexagon or a $4$-dimensional cross polytope if $p=4$.
	\end{itemize}
	Moreover, each of the listed seeds supports a unique fan-giving characteristic map with an optimal partition, and thus $\lambda$ is obtained by sequential canonical extensions.
\end{corollary}

Before proving Corollary~\ref{corollary:min_comp}, we need some preliminary results about the join of two simplicial complexes $K$ and $L$ on $\{1, 2, \ldots, m_1\}$ and $\{m_1+1, m_1+2, \ldots, m_2\}$, respectively.
Any minimal non-face of $K*L$ is a minimal non-face of either  $K$ or $L$.
In addition, a (mod~$2$) characteristic map $\lambda$ over $K*L$ can be represented as \begin{align*}
	\lambda = \begin{bmatrix}
		\lambda_{11} & \lambda_{12} \\
		\lambda_{21} & \lambda_{22}
	\end{bmatrix},
\end{align*}
where $\lambda_{11}$ is a characteristic map over $K$, and $\lambda_{22}$ is a characteristic map over $L$.
Furthermore, by D-J equivalence, if $\dim(K)=n_1-1$, then we can assume that the first $n_1$ columns of $\lambda_{11}$ is an identity matrix, and the first $n_1$ columns of $\lambda_{21}$ is zeros, and a similar argument holds for $\lambda_{12}$ and $\lambda_{22}$.

\begin{proof}[Proof of Corollary~\ref{corollary:min_comp}]
 	We first consider the boundary of any cross polytope.
	Recall that the suspension of $K$ is the join of a $0$-sphere $S^0$ and $K$.
	Then any characteristic map $\lambda$ over $S^0 \ast S^0$ is of the form
	\begin{align*}
		\begin{bmatrix}
			1 & \pm 1 & 0 & b \\
			0 & a & 1 & \pm 1
		\end{bmatrix}
	\end{align*} up to D-J equivalence if the vertex set of each $S^0$'s are $\{1, 2\}$ and $\{3, 4\}$.
	Hence $(S^0 \ast S^0, \lambda)$ has an optimal partition if and only if
	\begin{align*}
		\lambda = \begin{bmatrix}
			1 & -1 & 0 & 0 \\
			0 & 0 & 1 & -1
		\end{bmatrix}.
	\end{align*} up to D-J equivalence.
	Since $S^0 \ast \dots \ast S^0$ has minimal non-faces $\{1,2\}, \{3,4\}, \ldots$, a similar argument holds, so $(S^0 \ast \dots \ast S^0, \lambda)$
	has an optimal partition if and only if $\lambda$ is a block diagonal matrix whose block diagonal elements are all $\begin{bmatrix}
		1 & -1
	\end{bmatrix}$
	up to D-J equivalence.
	Note that any toric colorable seed of Picard number~$1$ is $S^0$, and that of Picard number~$2$ is $S^0 \ast S^0$.
	
	For Picard number~$3$, there are two toric colorable seeds that are not the boundary of a $3$-dimensional cross polytope:
	\begin{enumerate}
		\item the boundary of a pentagon.
		\item the boundary of a $4$-dimensional cyclic polytope $C^4(7)$ with $7$ vertices.
	\end{enumerate}
	For (1), one can check that there is no partition of the vertex set consisting of minimal non-faces.
	For (2), there is no fan-giving characteristic map from \cite{Choi-Park2016}.
	
	Finally, for Picard number~$4$, we use the list of toric colorable seeds obtained in Theorem~\ref{thm:main}.
	Similarly to the classification of toric colorable PL~spheres, we approach it from the mod~$2$ characteristic map perspective.
	Suppose that $P$ is an optimal partition for $(K, \lambda)$.
	Then for any $\{v_1, \ldots, v_k \} \in P$, the sum $\lambda(v_1) + \dots + \lambda(v_k)$ is also zero in mod~$2$.
	For a mod~$2$ characteristic map $\lambda^\R$ over $K$, we call a partition $P$ of $[m]$ \emph{weakly  optimal} if $P$ satisfies the condition of optimal partition with a mod~$2$ characteristic map $\lambda^\R$ instead of a characteristic map.
	Since $K$ has finitely many mod~$2$ characteristic maps, we can investigate all possibilities.
	For an $(n-1)$-dimensional regular seed, there is no partition consisting of minimal non-faces if $n$ is $10$ or $11$.
	Moreover, the boundary of the hexagon is the only Picard number~$4$ toric colorable regular seed which has a weakly optimal partition.
	More precisely, consider the boundary $K$ of the hexagon whose facets are $\{1,2\}, \{1,6\}, \{2,3\}, \{3,4\}, \{4,5\}, \{5,6\}$.
	Then there are four partitions $P_1 = \{\{1, 3\}, \{2, 5\}, \{4, 6\}\}$, $P_2 = \{\{1, 4\}, \{2, 6\}, \{3, 5\}\}$, $P_3=\{\{1, 5\}, \{2, 4\}, \{3, 6\}\}$, and $P_4 = \{\{1, 4\}, \{2, 5\}, \{3, 6\}\}$.
	Suppose that $P_1$ is an optimal partition of $(K, \lambda)$.
	We can choose a D-J class with $\lambda(1) = \begin{bmatrix}
		1 \\
		0
	\end{bmatrix}$, $\lambda(2) = \begin{bmatrix}
		0 \\
		1
	\end{bmatrix}$.
	Then by optimality of $P_1$, $\lambda(3) = \begin{bmatrix}
		-1 \\
		0
	\end{bmatrix}$, $\lambda(5) = \begin{bmatrix}
		0 \\
		-1
	\end{bmatrix}$.
	By fan-givingness of $\lambda$, $\lambda(4) = \begin{bmatrix}
		-1 \\
		-1
	\end{bmatrix}$.
	Again, by optimality of $P_1$, $\lambda(6) = \begin{bmatrix}
		1 \\
		1
	\end{bmatrix}$, but $\lambda$ does not yield a fan.
	Hence there is no $\lambda$ such that $P_1$ is an optimal partition for $(K, \lambda)$.
	The partitions $P_2$ and $P_3$ are the same as $P_1$ since they are obtained after rotating the labels of the hexagon.
	For $P_4$, by using optimality and fan-givingness, one can similarly show that $P_4$ is an optimal partition of $(K, \lambda)$ if and only if $\lambda = \begin{bmatrix}
		1 & 0 & -1 & -1 & 0 & 1 \\
		0 & 1 & 1 & 0 & -1 & -1
	\end{bmatrix}$.
	
	There are in addition three suspended seeds of Picard number~$4$ except the boundary of a $4$-dimensional cross polytope:
	\begin{enumerate}
		\item the suspension of the boundary of a pentagon, and
		\item the suspension of the boundary of $C^4(7)$.
	\end{enumerate}
	For these cases, by the construction of suspension, there is no optimal partition from similar reasons as for the Picard number~$3$ case.
\end{proof}

\section*{Acknowledgements}
The authors are very grateful to Axel Bacher, who introduced the last-named author to CUDA programming allowing Algorithm~\ref{algorithm:main} to fit with GPU computing and terminate in a reasonable time.
The authors also extend their gratitude to an anonymous reviewer for his/her invaluable comments on the originally submitted manuscript, particularly for suggesting a structural change in the introduction and to include additional data to enrich the article's content.

\bibliographystyle{amsplain}
\providecommand{\bysame}{\leavevmode\hbox to3em{\hrulefill}\thinspace}
\providecommand{\MR}{\relax\ifhmode\unskip\space\fi MR }
\providecommand{\MRhref}[2]{%
	\href{http://www.ams.org/mathscinet-getitem?mr=#1}{#2}
}
\providecommand{\href}[2]{#2}

\end{document}